\documentclass[12pt,a4paper,twoside]{article}
\usepackage[colorlinks=true,linkcolor=blue,citecolor=blue]{hyperref}
\usepackage{amsmath,mathrsfs,multicol,amsfonts,mathtools,amsthm, bm,fancyhdr,lastpage,xcolor,titlesec,cleveref}
\usepackage[english]{babel}
\usepackage{nicematrix}
\usepackage{amssymb}
\usepackage{orcidlink}
\definecolor{myeditcolor}{named}{blue}
\usepackage[bottom,multiple]{footmisc}
\usepackage[top=2.5cm,bottom=2.5cm,left=2.5cm,right=2.5cm]{geometry}
\titleformat{\section}[block]{\bfseries\large}{\thesection. }{2pt}{}
\theoremstyle{definition}
\newtheorem{theorem}{Theorem}[section]

\newtheorem{proposition}{Proposition}[section]
\newtheorem{corollary}[theorem]{Corollary}
\newtheorem{lemma}[theorem]{Lemma}

\newtheorem{remark}[theorem]{Remark}

\begin{document}
\thispagestyle{empty}
\begin{center}
\noindent\textbf{{\Large Classification of left invariant Riemannian metrics on  nonunimodular 4-dimensional Lie groups}}
\end{center}
\begin{center} \noindent Malika Ait Ben Haddou\footnote{m.aitbenhaddou@umi.ac.ma} and Youssef Ayad\orcidlink{0000-0002-8404-9455}\footnote{corresponding author: youssef.ayad@edu.umi.ac.ma}\\
\small{\textit{$^{1, 2}$Faculty of sciences, Moulay Ismail University of Mekn\`{e}s}}\\
\small{\textit{B.P. 11201, Zitoune, Mekn\`{e}s}, Morocco}
\end{center}
\begin{abstract}
We classify, up to automorphism, left invariant Riemannian metrics on 4-dimensional simply connected nonunimodular Lie groups. This is equivalent to classifying, up to automorphism, inner products on 4-dimensional nonunimodular Lie algebras.
\end{abstract}
\begin{center}
\end{center}
\begin{center}
\end{center}
\section{Introduction}
\hspace*{3mm}     The primary objective of this paper is to classify, up to automorphism, left invariant Riemannian metrics on 4-dimensional, simply connected nonunimodular Lie groups. This task is equivalent to classifying, up to automorphism, inner products on 4-dimensional nonunimodular Lie algebras. The classification of left invariant Riemannian metrics on simply connected unimodular Lie groups is credited to Ha and Lee \cite{ha2009left} in dimension three, and to Van Thuong \cite{van2017metrics} in dimension four. Additionally, Boucetta and Chakkar \cite{boucetta2022moduli} provided a classification of Lorentzian metrics on three-dimensional Lie groups in the unimodular case, while Ha and Lee \cite{ha2023left} addressed the nonunimodular case. Also, the classification of left invariant Lorentzian metrics on nilpotent four-dimensional Lie groups is given in \cite{bokan2015lorentz}.

The study of left invariant metrics on Lie groups has a long history and has been influenced by several fundamental contributions. In his classical paper, Milnor \cite{milnor1976curvatures} investigated the curvature properties of left invariant metrics on Lie groups and established a framework that has become a standard reference in the field. Earlier, Jensen \cite{jensen1969homogeneous} studied four-dimensional homogeneous Einstein spaces, obtaining important results on Einstein metrics in low dimensions. More recently, Karki and Thompson \cite{karki2016four} analyzed four-dimensional Einstein Lie groups, while Calvaruso and Zaeim \cite{calvaruso2013four} investigated the geometry of four-dimensional Lorentzian Lie groups. These works highlight the continued interest in the geometry of invariant metrics on low-dimensional Lie groups and provide further motivation for the present study.

For left invariant metrics of arbitrary signature, a classification on 4-dimensional solvable Lie groups was obtained in \cite{vsukilovic2020classification}. However, the classification of left invariant Riemannian metrics on nonunimodular Lie groups of dimension four still requires a more detailed analysis. This serves as a motivating factor for us to address this significant problem. An interesting question that arises after classifying these metrics is the characterization of their isometry groups, which hold significant importance in both mathematics and physics. Elements of the isometry group act by preserving the fundamental geometric structures of the Lie group, such as geodesics, the Levi–Civita connection, and curvature. The problem of the determination of the isometry groups of metrics on Lie groups has been fully resolved in dimension three by the following studies \cite{ana2022isometry,ha2012isometry,boucetta2022isometry,shin1997isometry}. In dimension four, this problem is resolved in some interesting cases by the following works \cite{aitbenhaddou2024oscillator,ayad2024isometry,vsukilovic2017isometry}.  Let $\widetilde{\mathfrak{M}}$ be the set of inner products on a Lie algebra $\mathfrak{g}$. There exists a natural action of the automorphism group $\operatorname{Aut}(\mathfrak{g})$ on $\widetilde{\mathfrak{M}}$ given by, for $A \in \operatorname{Aut}(\mathfrak{g})$,
$$A^{\ast}\langle u, v\rangle = \langle A^{-1}u, A^{-1}v\rangle \quad \forall u, v \in \mathfrak{g},$$
where $A^{\ast}\langle ., . \rangle$ is the pullback of the inner product $\langle ., . \rangle$ under $A$, and hence $\langle ., . \rangle$ and $A^{\ast}\langle ., . \rangle$ are isometric. The moduli space of left invariant metrics on a Lie group $G$ with Lie algebra $\mathfrak{g}$ is the orbit space of the action of $\operatorname{Aut}(\mathfrak{g})$ on $\widetilde{\mathfrak{M}}$ given by \cite{kodama2011space}
$$\mathfrak{M} = \operatorname{Aut}(\mathfrak{g}) \backslash \widetilde{\mathfrak{M}}.$$
This is the space of all left invariant metrics up to automorphism.\\
In \cite{van2017metrics}, Van Thuong provided a crucial method for classifying inner products on 4-dimensional Lie algebras, focusing on the unimodular case.  In this paper, we extend this classification to the nonunimodular case. A comprehensive list of all 4-dimensional nonunimodular Lie algebras is given in Table~\ref{tab:lie_algebras} \cite{Mubarakzyanov,kremlev2010signature}.

\begin{table}[htbp]
	\caption{4-dimensional nonunimodular Lie algebras}
	\label{tab:lie_algebras}
	\centering
	\begin{tabular}{|p{3.5cm}|p{10cm}|} \hline
		Lie algebra  & Nonzero commutators  \\ \hline
		$ \operatorname{A}_2 \oplus \operatorname{2A}_1 $ & $[e_1, e_2] = e_2$ \\ \hline
		$ \operatorname{2A}_2 $ &  $[e_1, e_2] = e_2, [e_3, e_4] = e_4$ \\ \hline
		$ \operatorname{A}_{3, 2} \oplus \operatorname{A}_1 $ & $[e_1, e_3] = e_1, [e_2, e_3] = e_1 + e_2$ \\ \hline
		$ \operatorname{A}_{3, 3} \oplus \operatorname{A}_1 $ & $[e_1, e_3] = e_1, [e_2, e_3] = e_2$ \\ \hline
		$\operatorname{A}_{3, 5}^{\alpha} \oplus \operatorname{A}_1, \newline
		0 < \left| \alpha \right|  < 1 $  & $[e_1, e_3] = e_1, [e_2, e_3] = \alpha e_2$ \\ \hline
		$ \operatorname{A}_{3, 7}^{\alpha} \oplus \operatorname{A}_1, \alpha > 0 $ &$[e_1, e_3] = \alpha e_1 - e_2, [e_2, e_3] = e_1 + \alpha e_2$ \\ \hline
		$ \operatorname{A}_{4, 2}^{\alpha}, \alpha \neq 0,\newline
		\alpha \neq -2 $ & $[e_1, e_4] = \alpha e_1, [e_2, e_4] = e_2, [e_3, e_4] = e_2 + e_3$ \\ \hline
		$ \operatorname{A}_{4, 3} $ & $[e_1, e_4] = e_1, [e_3, e_4] = e_2$ \\ \hline
		$ \operatorname{A}_{4, 4} $ & $[e_1, e_4] = e_1, [e_2, e_4] = e_1 + e_2, [e_3, e_4] = e_2 + e_3$  \\ \hline
		$ \operatorname{A}_{4, 5}^{\alpha, \beta}, \alpha\beta \neq 0, \newline
		-1 \leq \alpha \leq \beta \leq 1,\newline
		\alpha + \beta \neq -1 $ & $\newline
		[e_1, e_4] = e_1, [e_2, e_4] = \alpha e_2, [e_3, e_4] = \beta e_3$  \\ \hline
		$ \operatorname{A}_{4, 6}^{\alpha, \beta}, \alpha \neq 0, \newline
		\beta \geq 0, \alpha \neq -2\beta $ & $[e_1, e_4] = \alpha e_1, [e_2, e_4] = \beta e_2 - e_3, [e_3, e_4] = e_2 + \beta e_3$  \\ \hline	
		$ \operatorname{A}_{4, 7} $ & $[e_2, e_3] = e_1, [e_1, e_4] = 2e_1, [e_2, e_4] = e_2, [e_3, e_4] = e_2 + e_3$  \\ \hline
		$ \operatorname{A}_{4, 9}^{\beta}, -1 < \beta \leq 1 $ & $[e_2, e_3] = e_1, [e_1, e_4] = (1 + \beta)e_1, [e_2, e_4] = e_2, \newline
		[e_3, e_4] = \beta e_3$  \\ \hline
		$ \operatorname{A}_{4, 11}^{\alpha}, \alpha > 0 $ & $[e_2, e_3] = e_1, [e_1, e_4] = 2\alpha e_1, [e_2, e_4] = \alpha e_2 - e_3, \newline
		[e_3, e_4] = e_2 + \alpha e_3$  \\ \hline
		$ \operatorname{A}_{4, 12} $ & $[e_1, e_3] = e_1, [e_2, e_3] = e_2, [e_1, e_4] = -e_2, [e_2, e_4] = e_1$  \\ \hline
	\end{tabular}
\end{table}

In Table~\ref{tab:moduli_dimensions}, we summarize the dimensions of the moduli space of left-invariant Riemannian metrics on nonunimodular 4-dimensional Lie groups obtained in this paper and we compare them with those obtained in \cite{christodoulakis2003automorphisms}.

Note that for the Lie algebra $\operatorname{A}_{4, 2}^{\alpha}$, we distinguish the following two cases:
\begin{itemize}
	\item $\operatorname{A}_{4, 2}^{\alpha}$, where $\alpha \neq (0,1)$,
	\item $\operatorname{A}_{4, 2}^{1}$.
\end{itemize}

The automorphisms of these Lie algebras were corrected in \cite{christodoulakis2003automorphisms}, and the corrected results are presented in \cite{christodoulakis2003corrigendum}.

In most cases, the results obtained are consistent with those presented in \cite{christodoulakis2003automorphisms} regarding the possible reducible forms of $\gamma_{\alpha\beta}$; however, in this paper we provide a detailed proof for each case.
\begin{table}[htbp]
	\caption{Dimensions of the moduli space $\mathfrak{M}$}
	\label{tab:moduli_dimensions}
	\centering
	\begin{tabular}{|p{3.5cm}|p{3.5cm}|p{3.5cm}|} \hline
		Lie algebra  &  $\operatorname{dim}(\mathfrak{M})$ & $\operatorname{dim}(\mathfrak{M})$ in \cite{christodoulakis2003automorphisms} \\ \hline
		$ \operatorname{A}_2 \oplus \operatorname{2A}_1 $ & 2 & same \\ \hline
		$ \operatorname{2A}_2 $ &  6 & same \\ \hline
		$ \operatorname{A}_{3, 2} \oplus \operatorname{A}_1 $ & 5 & 4 \\ \hline
		$ \operatorname{A}_{3, 3} \oplus \operatorname{A}_1 $ & 3 & 2 \\ \hline
		$\operatorname{A}_{3, 5}^{\alpha} \oplus \operatorname{A}_1$  & 5 & 4 \\ \hline
		$ \operatorname{A}_{3, 7}^{\alpha} \oplus \operatorname{A}_1$ & 5 & 4 \\ \hline
		$ \operatorname{A}_{4, 2}^{\alpha}, \alpha \neq (0, 1)$ & 4 & same \\ \hline
		$\operatorname{A}_{4, 2}^{1}$ & 3 & 2 \\ \hline
		$ \operatorname{A}_{4, 3} $ & 4 & same \\ \hline
		$ \operatorname{A}_{4, 4} $ & 4 & same  \\ \hline
		$ \operatorname{A}_{4, 5}^{\alpha, \beta}, \alpha\beta \neq 0$ & 4 & same  \\ \hline
		$ \operatorname{A}_{4, 6}^{\alpha, \beta}, \alpha \neq 0 $ & 4 & same  \\ \hline	
		$ \operatorname{A}_{4, 7} $ & 5 & same  \\ \hline
		$ \operatorname{A}_{4, 9}^{\beta} $ & 5 & same \\ \hline
		$ \operatorname{A}_{4, 11}^{\alpha} $ & 5 & same  \\ \hline
		$ \operatorname{A}_{4, 12} $ & 6 & same  \\ \hline
	\end{tabular}
\end{table}

\section{Preliminaries}
The space of inner products on $\mathfrak{g}$ is the space of symmetric, positive definite matrices denoted $\operatorname{S}_n$. Consider the following notations\\
$\operatorname{Tsup}_n$ := the group consisting of upper triangular matrices with positive diagonal entries.\\
$\operatorname{Tinf}_n$ := the group consisting of lower triangular matrices with positive diagonal entries.\\
$\operatorname{D}_n^{+}$ := the group consisting of diagonal matrices with positive elements.
\begin{proposition}
The following map is bijective
$$ \begin{array}{rcl}
\psi : \operatorname{Tsup}_{n}&\longrightarrow& \operatorname{S}_{n}\\
B &\longmapsto& (B^{-1})^T(B^{-1})
\end{array} $$
\end{proposition}
\begin{proof}
It is clear that the map $\psi$ is well defined. For the injectivity, let $B, C \in \operatorname{Tsup}_{n}$ such that $\psi(B) = \psi(C)$. This is equivalent to the following equality $B^{-1}C = (C^{-1}B)^{T}$. Put $M = C^{-1}B$, then 
$B^{-1}C = (C^{-1}B)^{T} \Leftrightarrow M^{-1} = M^{T}$. Since $\operatorname{Tsup}_{n}$ is a group, then $M, M^{-1} \in \operatorname{Tsup}_{n}$. Since $M^{-1} = M^{T} \in \operatorname{Tinf}_n$, then $M \in \operatorname{Tsup}_n \cap \operatorname{Tinf}_n = \operatorname{D}_n^{+}$. Hence we have $M^{-1} = M^{T} = M$, then $M = I_n$. This implies that $B = C$ and $\psi$ is injective.\\
For the surjectivity, let $A \in \operatorname{S}_{n}$ and consider the inner product on $\mathbb{R}^{n}$ defined by 
$$\langle u, v\rangle = u^{T}Av, \forall u, v \in \mathbb{R}^{n}.$$
Let $\mathcal{B} = \left\lbrace e_1, ..., e_n\right\rbrace$ be the canonical basis of $\mathbb{R}^{n}$. The Gram-Schmidt procedure produces an orthonormal basis $\mathscr{B} = \left\lbrace v_1, ..., v_n\right\rbrace$ of the Euclidean space $\left( \mathbb{R}^{n}, \langle ., .\rangle\right)$ such that
$$[v_1, ..., v_n] = X \in Tsup_n.$$
This means that $Xe_i = v_i$. Since $\mathscr{B} = \{v_1, ..., v_n\}$ is an orthonormal basis of $(\mathbb{R}^n, \langle ., . \rangle)$, then
$$\operatorname{Mat}\left(\langle ., .\rangle , \mathscr{B}\right) =X^TAX = I_n.$$
Thus, $A = (X^{-1})^{T}(X^{-1}) = \psi(X)$, so $\psi$ is surjective and therefore bijective.
\end{proof}
\begin{corollary} (proposition 1.4 in\cite{van2017metrics})
The set of inner products on $\mathfrak{g}$ is $(n^2 + n)/2$-dimensional, and can be identified with the set of upper triangular matrices with positive diagonal entries.
\end{corollary}
\begin{remark}
The moduli space of left invariant metrics given by $\mathfrak{M} = \operatorname{Aut}(\mathfrak{g}) \backslash \widetilde{\mathfrak{M}}$ is identified with the double coset space $\operatorname{Aut}(\mathfrak{g})\backslash \operatorname{GL}(n, \mathbb{R})/\operatorname{O}(n, \mathbb{R})$. See \cite{van2017metrics}.
\end{remark}
Let us recall the method of classification of inner products on Lie algebras $\mathfrak{g}$ with $\operatorname{dim}(\mathfrak{g}) = 4$ given in section 2 of the following paper \cite{van2017metrics}. The classification was done by the following steps
\begin{enumerate}
\item Fixing a basis for $\mathfrak{g}$, calculate $\operatorname{Aut}(\mathfrak{g}) \subset \operatorname{GL}(4, \mathbb{R})$.
\item Consider an orthonormal basis for $\mathfrak{g}$, defined by an upper triangular matrix with positive diagonal entries :
\begin{equation} \label{basis}
B' = \begin{bmatrix}
b_{11} & b_{12} & b_{13} & b_{14}\\
0 & b_{22} & b_{23} & b_{24}\\
0 & 0 & b_{33} & b_{34}\\
0 & 0 & 0 & b_{44}\\
\end{bmatrix}, \quad [B'] \in \operatorname{GL}(4, \mathbb{R})/\operatorname{O}(4, \mathbb{R}).
\end{equation}
\item By left multiplication by some $A \in \operatorname{Aut}(\mathfrak{g})$, reduce $B'$ to a simpler upper triangular matrix $AB'$.
\item Note that $[B'], [B''] \in \operatorname{GL}(4, \mathbb{R})/\operatorname{O}(4, \mathbb{R})$ lie in the same $\operatorname{Aut}(\mathfrak{g})$ orbit if and only if $AB' = B''U$, for some $U \in \operatorname{O}(4, \mathbb{R})$, or
$$(B'')^{-1}AB' \in \operatorname{O}(4, \mathbb{R}).$$
\end{enumerate}
The following lemmas are crucial as they greatly simplify our main task of classifying upper triangular matrices with positive diagonal entries. The proofs of these lemmas are given in the Van Thuong paper \cite{van2017metrics}.
\begin{lemma} \label{lemma2.2}
Suppose for any $x \in \mathbb{R}^{3}$, there is $A \in \operatorname{Aut}(\mathfrak{g})$ which is block upper triangular an upper $3 \times 3$ block $\bar{A}$: $\begin{bmatrix}
\bar{A} & x\\
0 & a
\end{bmatrix}$. Then we may assume our orthonormal basis in (\ref{basis}) is block diagonal with upper $3 \times 3$ block: $\begin{bmatrix}
\bar{B} & 0\\
0 & b_{44}
\end{bmatrix}$.
\end{lemma}
\begin{lemma} \label{lemma2.3}
Let $B$ and $C$ be two orthonormal bases, both block diagonal with upper $3 \times 3$ block. Then for any $A \in \operatorname{Aut}(\mathfrak{g})$ which is block triangular with upper $3 \times 3$ block, $C^{-1}AB \in \operatorname{O}(4, \mathbb{R})$ forces $A$ to be block diagonal with upper $3 \times 3$ block.
\end{lemma}
\begin{lemma} \label{lemma2.4}
Let $M = \operatorname{GL}(2, \mathbb{R})/\operatorname{O}(2, \mathbb{R})$. Consider the subgroup of $\operatorname{GL}(2, \mathbb{R})$ isomorphic to $\left( \mathbb{R}^{\times}\right)^{2} \rtimes \mathbb{Z}_2$:
$$\left\lbrace \begin{bmatrix}
a & 0\\
0 & d
\end{bmatrix} |\quad a, d \neq 0\right\rbrace \cup \left\lbrace \begin{bmatrix}
0 & b\\
c & 0
\end{bmatrix} |\quad b, c \neq 0\right\rbrace.$$
Then $\left( \mathbb{R}^{\times}\right)^{2} \rtimes \mathbb{Z}_2\backslash M \cong \left\lbrace \begin{bmatrix}
1 & x\\
0 & 1
\end{bmatrix} |\quad x \geq 0\right\rbrace.$
\end{lemma}
\begin{lemma} \label{lemma2.5}
Let $M = \operatorname{GL}(2, \mathbb{R})/\operatorname{O}(2, \mathbb{R})$. Then identifying $\mathbb{R}^{+} = \left\lbrace \begin{bmatrix}
a & 0\\
0 & a
\end{bmatrix} |\quad a \in \mathbb{R}^{+}\right\rbrace.$\\
$\mathbb{R}^{+} \times \operatorname{O}(2, \mathbb{R})$ acts on $M$ by left multiplication, and 
$$\left( \mathbb{R}^{+} \times \operatorname{O}(2, \mathbb{R})\right)\backslash M \cong \left\lbrace \begin{bmatrix}
1 & 0\\
0 & x
\end{bmatrix} |\quad 0 < x \leq 1\right\rbrace,$$
where $\mathbb{R}^{+}\backslash M$ is identified with the upper half plane.
\end{lemma}
Now we are at the point of starting our classification of inner products on 4-dimensional nonunimodular Lie algebras. We begin with the decomposable ones.
\section{Decomposable nonunimodular 4-dimensional Lie algebras}
In this section, we classify the left invariant Riemannian metrics on all the decomposable Lie algebras $\mathfrak{g}$ from Table~\ref{tab:lie_algebras}.
\subsection{The Lie algebra $\operatorname{A}_2 \oplus \operatorname{2A}_1$}
The Lie algebra $\mathfrak{g} = \operatorname{A}_2 \oplus \operatorname{2A}_1$ has a basis $\mathcal{B} = \left\lbrace e_1, e_2, e_3, e_4\right\rbrace$ such that the only nonzero bracket is $[e_1, e_2] = e_2$.\\
The automorphism group of the Lie algebra $\operatorname{A}_2 \oplus \operatorname{2A}_1$ consists of elements of $\operatorname{GL}(4, \mathbb{R})$ of the form \cite{christodoulakis2003automorphisms}
$$\begin{bmatrix}
1 & 0 & 0 & 0\\
a_5 & a_6 & 0 & 0\\
a_9 & 0 & a_{11} & a_{12}\\
a_{13} & 0 & a_{15} & a_{16}
\end{bmatrix}.
$$
\begin{theorem}
Every metric on $\operatorname{A}_2 \oplus \operatorname{2A}_1$ is equivalent to one defined by an orthonormal basis
$$X_1 = b_{11}e_1, \quad X_2 = b_{12}e_1 + e_2, \quad X_3 = b_{13}e_1 + b_{23}e_2 + e_3, \quad X_4 = b_{24}e_2 + e_{4}
$$
where $b_{11} > 0$, \quad $b_{12}, b_{13} \geq 0$ and $b_{23}, b_{24} \in \mathbb{R}$.
\end{theorem}
\begin{proof}
Consider $B'$ to be an orthonormal basis of the form (\ref{basis}). We can find $A \in \operatorname{Aut}(\mathfrak{g})$, such that $AB' = B$, where
\begin{equation}
B = \begin{bmatrix} \label{form1}
b_{11} & b_{12} & b_{13} & b_{14}\\
0 & 1 & b_{23} & b_{24}\\
0 & 0 & 1 & 0\\
0 & 0 & 0 & 1\\
\end{bmatrix}, \quad b_{11} > 0.
\end{equation}
Now, we must decide if any two orthonormal bases $B = (b_{ij})$, $C = (c_{ij})$ of the form (\ref{form1}) are equivalent. We remark that the automorphism group of $\operatorname{A}_2 \oplus \operatorname{2A}_1$ contains the elements of the form $\begin{bmatrix}
1 & 0 & 0 & 0\\
0 & a_6 & 0 & 0\\
0 & 0 & a_{11} & a_{12}\\
0 & 0 & a_{15} & a_{16}\\
\end{bmatrix}$. 
Let $A$ be an automorphism of $\operatorname{A}_2 \oplus \operatorname{2A}_1$ of this form, we calculate
$$C^{-1}AB = \begin{bmatrix}
\frac{b_{11}}{c_{11}} & \frac{x}{c_{11}} & \frac{y}{c_{11}} & \frac{z}{c_{11}}\\
\\
0 & a_6 & u & v\\
\\
0 & 0 & a_{11} & a_{12}\\
\\
0 & 0 & a_{15} & a_{16}
\end{bmatrix}.$$
Where the elements $x, y, z, u$ and $v$ are given by
$$
\left\lbrace\begin{array}{lll}
x = b_{12} - a_{6}c_{12} \\
y = b_{13} - a_{6}c_{12}b_{23} + (c_{12}c_{23} - c_{13})a_{11} + (c_{12}c_{24} - c_{14})a_{15} \\
z = b_{14} - a_{6}c_{12}b_{24} + (c_{12}c_{23} - c_{13})a_{12} + (c_{12}c_{24} - c_{14})a_{16}\\
u = a_6b_{23} - a_{11}c_{23} - a_{15}c_{24}\\
v = a_6b_{24} - a_{12}c_{23} - a_{16}c_{24}
\end{array}\right.
$$
The matrix $C^{-1}AB$ is orthogonal precisely when the block $\begin{bmatrix}
a_{11} & a_{12}\\
a_{15} & a_{16}
\end{bmatrix}$ is orthogonal,\\ $b_{11} = c_{11}$, $x = y = z = u = v = 0$ and $a_{6} = \pm1$. We remark that
$$
\left\lbrace\begin{array}{lll}
y = b_{13} - c_{12}u - a_{11}c_{13} - a_{15}c_{14} \\
z = b_{14} - c_{12}v - a_{12}c_{13} - a_{16}c_{14}
\end{array}\right.
$$
Therefore we obtain that
\begin{eqnarray*}
\begin{bmatrix}
u\\
v
\end{bmatrix} = \begin{bmatrix}
0\\
0
\end{bmatrix} &\Leftrightarrow& a_6\begin{bmatrix}
b_{23}\\
b_{24}
\end{bmatrix} = \begin{bmatrix}
a_{11} & a_{15}\\
a_{12} & a_{16}
\end{bmatrix}\begin{bmatrix}
c_{23}\\
c_{24}
\end{bmatrix}\\
\begin{bmatrix}
y\\
z
\end{bmatrix} = \begin{bmatrix}
0\\
0
\end{bmatrix} &\Leftrightarrow& \hspace{0.48cm} \begin{bmatrix}
b_{13}\\
b_{14}
\end{bmatrix} = \begin{bmatrix}
a_{11} & a_{15}\\
a_{12} & a_{16}
\end{bmatrix}\begin{bmatrix}
c_{13}\\
c_{14}
\end{bmatrix}.
\end{eqnarray*}
We can choose an orthogonal matrix $\begin{bmatrix}
a_{11} & a_{12} \\
a_{15} & a_{16}
\end{bmatrix}$
such that its transpose maps the vector $(c_{13}, c_{14})$ to a point on the nonnegative $x$-axis. Therefore, we may assume that $b_{14} = 0$ and $b_{13} \geq 0$. However, with that choice, we cannot control the vector $a_6(b_{23}, b_{24})$. Still, by choice of $a_6$, we can ensure that $b_{12} \geq 0$.
\end{proof}
\begin{remark}
If we consider the Riemannian metrics as lower triangular matrices with positive diagonal entries, then we can reduce the dimension of the moduli space of the Lie algebra $\operatorname{A}_2 \oplus \operatorname{2A}_1$ to $\operatorname{dim}(\mathfrak{M}) = 2$. In fact, we can find an automorphism of  $\operatorname{A}_2 \oplus \operatorname{2A}_1$ such that
$$\begin{bmatrix}
1 & 0 & 0 & 0\\
a_5 & a_6 & 0 & 0\\
a_9 & 0 & a_{11} & a_{12}\\
a_{13} & 0 & a_{15} & a_{16}
\end{bmatrix}\begin{bmatrix}
b_{11} & 0 & 0 & 0\\
b_{21} & b_{22} & 0 & 0\\
b_{31} & b_{32} & b_{33} & 0\\
b_{41} & b_{42} & b_{43} & b_{44}
\end{bmatrix} = \begin{bmatrix}
\ast & 0 & 0 & 0\\
0 & 1 & 0 & 0\\
0 & \ast & 1 & 0\\
0 & \ast & 0 & 1
\end{bmatrix}.$$
Consider two metrics of the following form:
$$B = \begin{bmatrix}
b_{11} & 0 & 0 & 0\\
0 & 1 & 0 & 0\\
0 & b_{32} & 1 & 0\\
0 & b_{42} & 0 & 1
\end{bmatrix}, \qquad C = \begin{bmatrix}
c_{11} & 0 & 0 & 0\\
0 & 1 & 0 & 0\\
0 & c_{32} & 1 & 0\\
0 & c_{42} & 0 & 1
\end{bmatrix}.$$
Let $A = \begin{bmatrix}
1 & 0 & 0 & 0\\
0 & a_6 & 0 & 0\\
0 & 0 & a_{11} & a_{12}\\
0 & 0 & a_{15} & a_{16}\\
\end{bmatrix} \in \operatorname{Aut}\left(\operatorname{A}_2 \oplus \operatorname{2A}_1\right)$, then we calculate:
$$C^{-1}AB = \begin{bmatrix}
\frac{b_{11}}{c_{11}} & 0 & 0 & 0\\
\\
0 & a_6 & 0 & 0\\
\\
0 & x & a_{11} & a_{12}\\
\\
0 & y & a_{15} & a_{16}
\end{bmatrix},$$
where $x = -a_6c_{32} + a_{11}b_{32} + a_{12}b_{42}$ and $y = -a_6c_{42} + a_{15}b_{32} + a_{16}b_{42}$. The matrix $C^{-1}AB$ is orthogonal precisely when the block $\begin{bmatrix}
a_{11} & a_{12}\\
a_{15} & a_{16}
\end{bmatrix}$ is orthogonal, $b_{11} = c_{11}$, $a_6 = \pm 1$ and $x = y = 0$. This implies that
$$a_6\begin{bmatrix}
c_{32}\\
c_{42}
\end{bmatrix} = \begin{bmatrix}
a_{11} & a_{12}\\
a_{15} & a_{16}
\end{bmatrix}\begin{bmatrix}
b_{32}\\
b_{42}
\end{bmatrix},$$
which means that $(b_{32}, b_{42})$ and $(c_{32}, c_{42})$ are in the same orbit of the natural action of $\operatorname{O}(2, \mathbb{R})$ on $\mathbb{R}^2$. Therefore we can put $b_{42} = 0$ and $b_{32} \geq 0$. Thus, the metric is equivalent to the following:
$$B = \begin{bmatrix}
b_{11} & 0 & 0 & 0\\
0 & 1 & 0 & 0\\
0 & b_{32} & 1 & 0\\
0 & 0 & 0 & 1
\end{bmatrix}.$$
Consequently $\operatorname{dim}(\mathfrak{M}) = 2$.
\end{remark}
\subsection{The Lie algebra $\operatorname{2A}_2$}
The Lie algebra $\operatorname{2A}_2$ has a basis $\mathcal{B} = \left\lbrace e_1, e_2, e_3, e_4\right\rbrace$ such that the only nonzero brackets are 
$$[e_1, e_2] = e_2, \qquad [e_3, e_4] = e_4.$$
The automorphism group of the Lie algebra $\operatorname{2A}_2$ is formed by elements of $\operatorname{GL}(4, \mathbb{R})$ of the form \cite{christodoulakis2003automorphisms}
$$\begin{bmatrix}
1 & 0 & 0 & 0\\
a_5 & a_6 & 0 & 0\\
0 & 0 & 1 & 0\\
0 & 0 & a_{15} & a_{16}
\end{bmatrix} \quad \text{or}\quad \begin{bmatrix}
0 & 0 & 1 & 0\\
0 & 0 & a_{7} & a_{8}\\
1 & 0 & 0 & 0\\
a_{13} & a_{14} & 0 & 0
\end{bmatrix}.
$$
\begin{theorem}
Every metric on $\operatorname{2A}_2$ is equivalent to one defined by an orthonormal basis
$$X_1 = b_{11}e_1, \quad X_2 = b_{12}e_1 + e_2, \quad X_3 = b_{13}e_1 + b_{23}e_2 + b_{33}e_3, \quad X_4 = b_{14}e_1 + b_{24}e_2 + b_{34}e_3 + e_{4}
$$
where $b_{11}, b_{33} > 0, \quad b_{12}, b_{34} \geq 0, \quad b_{13}, b_{23}, b_{14},  b_{24} \in \mathbb{R}$.
\end{theorem}
\begin{proof}
Applying an upper triangular automorphism of $\operatorname{2A}_2$, every metric on $\operatorname{2A}_2$ is equivalent to one determined by orthonormal basis
\begin{equation} \label{form2}
 \begin{bmatrix}
b_{11} & b_{12} & b_{13} & b_{14}\\
0 & 1 & b_{23} & b_{24}\\
0 & 0 & b_{33} & b_{34}\\
0 & 0 & 0 & 1\\
\end{bmatrix}, \quad b_{11}, b_{33} > 0.
\end{equation}
Now, we must decide if any two orthonormal bases $B = (b_{ij})$, $C = (c_{ij})$ of the form (\ref{form2}) are equivalent. By choosing a diagonal automorphism $A$ of $\mathfrak{g}$, we obtain that
$$C^{-1}AB = \begin{bmatrix}
\frac{b_{11}}{c_{11}} & \frac{x}{c_{11}} & \frac{y}{c_{11}} & \frac{z}{c_{11}}\\
0 & a_6 & u & v\\
0 & 0 & \frac{b_{33}}{c_{33}} & \frac{w}{c_{33}}\\
0 & 0 & 0 & a_{16}
\end{bmatrix}.$$
Where the elements $x, y, z, u$ and $v$ are given by
$$
\left\lbrace\begin{array}{lll}
x = b_{12} - a_{6}c_{12} \\
y = b_{13} - a_6c_{12}b_{23} + (c_{12}c_{23} - c_{13})\frac{b_{33}}{c_{33}} \\
z = b_{14} - a_{6}c_{12}b_{24} + (c_{12}c_{23} - c_{13})\frac{b_{34}}{c_{33}} - \frac{\left(c_{12}c_{23}c_{34} - c_{12}c_{24}c_{33} - c_{13}c_{34} + c_{14}c_{33}\right)a_{16}}{c_{33}}\\
u = a_6b_{23} - c_{23}\frac{b_{33}}{c_{33}}\\
v = a_6b_{24} - c_{23}\frac{b_{34}}{c_{33}} + \frac{\left(c_{23}c_{34} - c_{24}c_{33}\right)a_{16}}{c_{33}}\\
w = b_{34} - a_{16}c_{34}
\end{array}\right.
$$
The matrix $C^{-1}AB$ is orthogonal precisely when $b_{11} = c_{11}$, $b_{33} = c_{33}$, $a_{6}$ and $a_{16}$ are equal to $\pm1$ and $x = y = z = u = v = 0$. This implies that
$$
\left\lbrace\begin{array}{lll}
b_{12} =  a_{6}c_{12} \\
a_6b_{23} = c_{23}\\
b_{34} = a_{16}c_{34}
\end{array}\right.
$$
By choice of $a_{6}$ and $a_{16}$, any orthonormal basis is equivalent to one with $b_{12}, b_{34} \geq 0$ and $b_{13}, b_{23}, b_{14}, b_{24} \in \mathbb{R}$.
\end{proof}
\begin{remark}
If we consider the Riemannian metrics as lower triangular matrices with positive diagonal entries, then we can reduce the dimension of the moduli space of the Lie algebra $\operatorname{2A}_2$ to $\operatorname{dim}(\mathfrak{M}) = 6$. Because we can find an automorphism of  $\operatorname{2A}_2$ such that
$$\begin{bmatrix}
1 & 0 & 0 & 0\\
a_5 & a_6 & 0 & 0\\
0 & 0 & 1 & 0\\
0 & 0 & a_{15} & a_{16}
\end{bmatrix}\begin{bmatrix}
b_{11} & 0 & 0 & 0\\
b_{21} & b_{22} & 0 & 0\\
b_{31} & b_{32} & b_{33} & 0\\
b_{41} & b_{42} & b_{43} & b_{44}
\end{bmatrix} = \begin{bmatrix}
\ast & 0 & 0 & 0\\
0 & 1 & 0 & 0\\
\ast & \ast & \ast & 0\\
\ast & \ast & 0 & 1
\end{bmatrix}.$$
\end{remark}
\subsection{The Lie algebra $\operatorname{A}_{3, 2} \oplus \operatorname{A}_1$}
The Lie algebra $\operatorname{A}_{3, 2} \oplus \operatorname{A}_1$ has a basis $\mathcal{B} = \left\lbrace e_1, e_2, e_3, e_4\right\rbrace$ such that the nonzero brackets are 
$$[e_1, e_3] = e_1, \quad [e_2, e_3] = e_1 + e_2.$$
The group of automorphisms of the Lie algebra $\operatorname{A}_{3, 2} \oplus \operatorname{A}_1$ is given by \cite{christodoulakis2003automorphisms}
$$\operatorname{Aut}\left( \operatorname{A}_{3, 2} \oplus \operatorname{A}_1 \right) = \left\lbrace \begin{bmatrix}
a_1 & a_2 & a_3 & 0\\
0 & a_1 & a_7 & 0\\
0 & 0 & 1 & 0\\
0 & 0 & a_{15} & a_{16}
\end{bmatrix} \right\rbrace  \subset \operatorname{GL}(4, \mathbb{R}).
$$
\begin{theorem}
Every metric on $\operatorname{A}_{3, 2} \oplus \operatorname{A}_1$ is equivalent to one defined by an orthonormal basis
$$X_1 = b_{11}e_1, \quad X_2 = e_2, \quad X_3 = b_{33}e_3, \quad X_4 = b_{14}e_1 + b_{24}e_2 + b_{34}e_3 + e_{4}
$$
where $b_{11}, b_{33} > 0, \quad b_{24}, b_{34} \geq 0, \quad b_{14} \in \mathbb{R}$.
\end{theorem}
\begin{proof}
By left multiplication by an upper triangular automorphism, every metric on $\operatorname{A}_{3, 2} \oplus \operatorname{A}_1$ is equivalent to one determined by orthonormal basis
\begin{equation} 
B = \begin{bmatrix} \label{form4}
b_{11} & 0 & 0 & b_{14}\\
0 & 1 & 0 & b_{24}\\
0 & 0 & b_{33} & b_{34}\\
0 & 0 & 0 & 1\\
\end{bmatrix}, \quad b_{11}, b_{33} > 0. 
\end{equation}
Now, we must decide if any two orthonormal bases $B = (b_{ij})$, $C = (c_{ij})$ of the form (\ref{form4}) are equivalent. By choosing a diagonal automorphism $A$ of $\mathfrak{g}$, we calculate
$$C^{-1}AB = \begin{bmatrix}
\frac{a_{1}b_{11}}{c_{11}} & 0 & 0 & \frac{a_1b_{14} - a_{16}c_{14}}{c_{11}}\\
\\
0 & a_{1} & 0 & a_1b_{24} - a_{16}c_{24}\\
\\
0 & 0 & \frac{b_{33}}{c_{33}} & \frac{b_{34} - a_{16}c_{34}}{c_{33}}\\
\\
0 & 0 & 0 & a_{16}
\end{bmatrix}.$$
The matrix $C^{-1}AB$ is orthogonal precisely when $a_1, a_{16} = \pm1$, $b_{11} = c_{11}$,  $b_{33} = c_{33}$ and
$$
\left\lbrace\begin{array}{lll}
a_1b_{14} = a_{16}c_{14} \\
a_1b_{24} = a_{16}c_{24}\\
b_{34} = a_{16}c_{34}
\end{array}\right.
$$
By choice of $a_1$ and $a_{16}$ we can put $b_{34}, b_{24} \geq 0$, but we have $b_{14} \in \mathbb{R}$.
\end{proof}
\begin{remark}
We cannot use the rotation group in this case because, if we consider an automorphism with $a_{15}=0$, the condition that $C^{-1}AB$ be orthogonal forces $a_2=a_3=a_7=0$. Therefore, we can only use a diagonal automorphism.
\end{remark}
\subsection{The Lie algebra $\operatorname{A}_{3, 3} \oplus \operatorname{A}_1$}
The Lie algebra $\operatorname{A}_{3, 3} \oplus \operatorname{A}_1$ has a basis $\mathcal{B} = \left\lbrace e_1, e_2, e_3, e_4\right\rbrace$ such that the nonzero brackets are 
$$[e_1, e_3] = e_1, \quad [e_2, e_3] = e_2.$$
The automorphisms of the Lie algebra $\operatorname{A}_{3, 3} \oplus \operatorname{A}_1$ are given by elements of $\operatorname{GL}(4, \mathbb{R})$ that take the form \cite{christodoulakis2003automorphisms}
$$\begin{bmatrix}
a_1 & a_2 & a_3 & 0\\
a_5 & a_6 & a_7 & 0\\
0 & 0 & 1 & 0\\
0 & 0 & a_{15} & a_{16}
\end{bmatrix}.
$$
\begin{theorem}
Every metric on $\operatorname{A}_{3, 3} \oplus \operatorname{A}_1$ is equivalent to one defined by an orthonormal basis
$$X_1 = e_1, \quad X_2 = e_2, \quad X_3 = b_{33}e_3, \quad X_4 = b_{14}e_1 + b_{34}e_3 + e_{4}
$$
where $b_{33} > 0, \quad b_{14}, b_{34} \geq 0$.
\end{theorem}
\begin{proof}
Applying an upper triangular automorphism, every metric on $\operatorname{A}_{3, 3} \oplus \operatorname{A}_1$ is equivalent to one determined by orthonormal basis
\begin{equation} \label{form5}
\begin{bmatrix}
1 & 0 & 0 & b_{14}\\
0 & 1 & 0 & b_{24}\\
0 & 0 & b_{33} & b_{34}\\
0 & 0 & 0 & 1\\
\end{bmatrix}, \quad b_{33} > 0. 
\end{equation}
We remark that the automorphism group of $\operatorname{A}_{3, 3} \oplus \operatorname{A}_1$ contains the elements of the form $\begin{bmatrix}
a_1 & a_2 & 0 & 0\\
a_5 & a_6 & 0 & 0\\
0 & 0 & 1 & 0\\
0 & 0 & 0 & a_{16}\\
\end{bmatrix}$. 
Let $A$ be an automorphism of $\operatorname{A}_{3, 3} \oplus \operatorname{A}_1$ of this form, given two orthonormal bases $B = (b_{ij})$ and $C = (c_{ij})$ of the form (\ref{form5}), we calculate
$$C^{-1}AB = \begin{bmatrix}
a_{1} & a_2 & 0 & a_1b_{14} + a_2b_{24} - a_{16}c_{14}\\
\\
a_5 & a_6 & 0 & a_5b_{14} + a_6b_{24} - a_{16}c_{24}\\
\\
0 & 0 & \frac{b_{33}}{c_{33}} & \frac{b_{34} - a_{16}c_{34}}{c_{33}}\\
\\
0 & 0 & 0 & a_{16}
\end{bmatrix}.$$
This matrix is orthogonal precisely when the upper $2 \times 2$ block is orthogonal, $a_{16} = \pm1$,  $b_{33} = c_{33}$, $b_{34} = a_{16}c_{34}$ and
$$\begin{bmatrix}
a_1 & a_2\\
a_5 & a_6
\end{bmatrix}\begin{bmatrix}
b_{14}\\
b_{24}
\end{bmatrix} = \pm\begin{bmatrix}
c_{14}\\
c_{24}
\end{bmatrix}, \quad \text{where}\; \begin{bmatrix}
a_1 & a_2\\
a_5 & a_6
\end{bmatrix} \in \operatorname{O}(2, \mathbb{R}).$$
This means that $(b_{14}, b_{24})$ and $(c_{14}, c_{24})$ are in the same orbit of the natural action of $\operatorname{O}(2, \mathbb{R})$ on $\mathbb{R}^2$. Therefore we can put $b_{24} = 0$ and $b_{14} \geq 0$. By choice of $a_{16}$ we can take $b_{34} \geq 0$.
\end{proof}
\subsection{The Lie algebra $\operatorname{A}_{3, 5}^{\alpha} \oplus \operatorname{A}_1$}
The Lie algebra $\operatorname{A}_{3, 5}^{\alpha} \oplus \operatorname{A}_1$ has a basis $\mathcal{B} = \left\lbrace e_1, e_2, e_3, e_4\right\rbrace$ such that the nonzero brackets are 
$$[e_1, e_3] = e_1,\quad [e_2, e_3] = \alpha e_2,\quad \text{where}\; 0 < |\alpha| < 1.$$
The automorphism group of the Lie algebra $\operatorname{A}_{3, 5}^{\alpha} \oplus \operatorname{A}_1$ is composed of elements in $\operatorname{GL}(4, \mathbb{R})$ of the form \cite{christodoulakis2003automorphisms}
$$\begin{bmatrix}
a_1 & 0 & a_3 & 0\\
0 & a_6 & a_7 & 0\\
0 & 0 & 1 & 0\\
0 & 0 & a_{15} & a_{16}
\end{bmatrix}.
$$
\begin{theorem}
Every metric on $\operatorname{A}_{3, 5}^{\alpha} \oplus \operatorname{A}_1$ is equivalent to one defined by an orthonormal basis
$$X_1 = e_1, \quad X_2 = b_{12}e_1 + e_2, \quad X_3 = b_{33}e_3, \quad X_4 = b_{14}e_1 + b_{24}e_2 + b_{34}e_3 + e_{4}
$$
where $b_{33} > 0, \quad b_{14}, b_{24}, b_{34} \geq 0, \quad b_{12} \in \mathbb{R}$.
\end{theorem}
\begin{proof}
By multiplying $B'$ by an upper triangular automorphism of $\operatorname{A}_{3, 5}^{\alpha} \oplus \operatorname{A}_1$, every metric on $\mathfrak{g} = \operatorname{A}_{3, 5}^{\alpha} \oplus \operatorname{A}_1$ is equivalent to one determined by orthonormal basis
\begin{equation} \label{form6}
\begin{bmatrix}
1 & b_{12} & 0 & b_{14}\\
0 & 1 & 0 & b_{24}\\
0 & 0 & b_{33} & b_{34}\\
0 & 0 & 0 & 1\\
\end{bmatrix}, \quad b_{33} > 0. 
\end{equation}
Given two orthonormal bases $B = (b_{ij})$ and $C = (c_{ij})$ of the form (\ref{form6}), we choose a diagonal automorphism $A$ of $\mathfrak{g}$ and we calculate
$$C^{-1}AB = \begin{bmatrix}
a_{1} & x & 0 & y\\
0 & a_6 & 0 & z\\
0 & 0 & \frac{b_{33}}{c_{33}} & \frac{u}{c_{33}}\\
0 & 0 & 0 & a_{16}
\end{bmatrix}.$$
Where the elements $x, y, z$ and $u$ are given by
$$
\left\lbrace\begin{array}{lll}
x = a_1b_{12} - a_6c_{12} \\
y = a_1b_{14} - a_6c_{12}b_{24} + (c_{12}c_{24} - c_{14})a_{16} \\
z = a_6b_{24} - a_{16}c_{24}\\
u = b_{34} - a_{16}c_{34}
\end{array}\right.
$$
The matrix $C^{-1}AB$ is orthogonal precisely when $a_1, a_6$ and $a_{16}$ are equal to $\pm1$,  $b_{33} = c_{33}$ and $x = y = z = u = 0$. Thus
$$
\left\lbrace\begin{array}{lll}
a_1b_{12} = a_6c_{12} \\
a_1b_{14} = a_{16}c_{14}  \\
a_6b_{24} = a_{16}c_{24} \\
b_{34} = a_{16}c_{34}
\end{array}\right.
$$
By choice of $a_{16}$, $a_6$ and $a_{1}$ we can put $b_{34}, b_{24}, b_{14} \geq 0$, but we have $b_{12} \in \mathbb{R}$.
\end{proof}
\subsection{The Lie algebra $\operatorname{A}_{3, 7}^{\alpha} \oplus \operatorname{A}_1$}
The Lie algebra $\operatorname{A}_{3, 7}^{\alpha} \oplus \operatorname{A}_1$ has a basis $\mathcal{B} = \left\lbrace e_1, e_2, e_3, e_4\right\rbrace$ such that the nonzero brackets are 
$$[e_1, e_3] = \alpha e_1 - e_2,\quad [e_2, e_3] = e_1 + \alpha e_2,\quad \text{where}\; \alpha > 0.$$
The automorphism group of the Lie algebra $\operatorname{A}_{3, 7}^{\alpha} \oplus \operatorname{A}_1$ consists of elements of $\operatorname{GL}(4, \mathbb{R})$ of the form \cite{christodoulakis2003automorphisms}
$$\begin{bmatrix}
a_1 & a_2 & a_3 & 0\\
-a_2 & a_1 & a_7 & 0\\
0 & 0 & 1 & 0\\
0 & 0 & a_{15} & a_{16}
\end{bmatrix}.
$$
\begin{theorem} \label{reasoning}
Every metric on $\operatorname{A}_{3, 7}^{\alpha} \oplus \operatorname{A}_1$ is equivalent to one defined by an orthonormal basis
$$X_1 = e_1, \quad X_2 = b_{22}e_2, \quad X_3 = b_{33}e_3, \quad X_4 = b_{14}e_1 + b_{24}e_2 + b_{34}e_{3} + e_{4}$$
$\hspace{2.3cm}  b_{33} > 0, \quad 0 < b_{22} < 1, \quad b_{14}, b_{24}, b_{34} \geq 0.$\\
or 
$$X_1 = e_1, \quad X_2 = b_{22}e_2, \quad X_3 = b_{33}e_3, \quad X_4 = b_{14}e_1 + b_{24}e_2 + b_{34}e_{3} + e_{4}
$$
$\hspace{2.3cm}  b_{33} > 0, \quad  0 < b_{22} < 1, \quad b_{14} > 0, \quad b_{24} < 0, \quad b_{34} \geq 0$.\\
or
$$X_1 = e_1, \quad X_2 = e_2, \quad X_3 = b_{33}e_3, \quad X_4 = b_{14}e_1 + b_{34}e_{3} + e_{4}
$$
$\hspace{2.3cm}  b_{33} > 0, \quad b_{14}, b_{34} \geq 0$.
\end{theorem}
\begin{proof}
By multiplying $B'$ by an upper triangular automorphism of $\mathfrak{g}$ and by using the lemma \ref{lemma2.5}, $B'$ is equivalent to the following
\begin{equation} \label{form7}
B = \begin{bmatrix}
1 & 0 & 0 & b_{14}\\
0 & b_{22} & 0 & b_{24}\\
0 & 0 & b_{33} & b_{34}\\
0 & 0 & 0 & 1\\
\end{bmatrix}, \quad 0 < b_{22} \leq 1, \quad b_{33} > 0, \quad b_{14}, b_{24},  b_{34} \in \mathbb{R}. 
\end{equation}
Given two orthonormal bases $B = (b_{ij})$ and $C = (c_{ij})$ of the form (\ref{form7}), let $A$ be an automorphism of $\mathfrak{g}$ of the form $A = \begin{bmatrix}
a_1 & a_2 & 0 & 0\\
-a_2 & a_1 & 0 & 0\\
0 & 0 & 1 & 0\\
0 & 0 & 0 & a_{16}
\end{bmatrix}$. We calculate
$$C^{-1}AB = \begin{bmatrix}
a_{1} & a_2b_{22} & 0 & a_1b_{14} + a_2b_{24} - a_{16}c_{14}\\
\\
\frac{-a_2}{c_{22}} & \frac{a_1b_{22}}{c_{22}} & 0 & \frac{-a_2b_{14} + a_1b_{24} - a_{16}c_{24}}{c_{22}}\\
\\
0 & 0 & \frac{b_{33}}{c_{33}} & \frac{b_{34} - a_{16}c_{34}}{c_{33}}\\
\\
0 & 0 & 0 & a_{16}
\end{bmatrix}.$$
This matrix is orthogonal precisely when the upper $2 \times 2$ block is orthogonal, $a_{16} = \pm1$, $b_{33} = c_{33}$, $b_{34} = a_{16}c_{34}$ and
$\begin{bmatrix}
a_1 & a_2\\
-a_2 & a_1
\end{bmatrix}\begin{bmatrix}
b_{14}\\
b_{24}
\end{bmatrix} = \pm\begin{bmatrix}
c_{14}\\
c_{24}
\end{bmatrix}.$ By choice of $a_{16}$ we take $b_{34} \geq 0$.\\
If $b_{22} \neq c_{22}$, then the upper $2 \times 2$ block cannot be orthogonal because we have\\ $0 < b_{22}, c_{22} \leq 1$. Thus, we assume that $b_{22} = c_{22}$.\\
First suppose that $b_{22} = c_{22} = 1$, then the equality 
$\begin{bmatrix}
a_1 & a_2\\
-a_2 & a_1
\end{bmatrix}\begin{bmatrix}
b_{14}\\
b_{24}
\end{bmatrix} = \pm\begin{bmatrix}
c_{14}\\
c_{24}
\end{bmatrix}$ means that $(b_{14}, b_{24})$ and $(c_{14}, c_{24})$ are in the same orbit of the natural action of $\operatorname{O}(2, \mathbb{R})$ on $\mathbb{R}^2$. From this we infer that when $b_{22} = 1$, we can put $b_{24} = 0$ and $b_{14} \geq 0$. This is the third metric in our theorem.\\
If $0 < b_{22} < 1$, then the condition that the upper $2 \times 2$ block of $C^{-1}AB$ is orthogonal implies that $a_2 = 0$. Thus
$ \begin{bmatrix}
b_{14}\\
b_{24}
\end{bmatrix} = \pm \begin{bmatrix}
c_{14}\\
c_{24}
\end{bmatrix}.$\\
If $c_{14}$ and $c_{24}$ have the same sign, we can take $b_{14}, b_{24} \geq 0$ (first metric in our theorem). If  $c_{14}$ and $c_{24}$ have different sign, we can take $b_{14} > 0$, $b_{24} < 0$ (second metric in our theorem).
\end{proof}
\section{Indecomposable nonunimodular 4-dimensional Lie algebras}
In this section, we classify the left invariant Riemannian metrics on all the indecomposable Lie algebras $\mathfrak{g}$ from Table~\ref{tab:lie_algebras}.
\subsection{The Lie algebra $\operatorname{A}_{4, 2}^{\alpha}$}
The Lie algebra $\operatorname{A}_{4, 2}^{\alpha}$ has a basis $\mathcal{B} = \left\lbrace e_1, e_2, e_3, e_4\right\rbrace$ such that the nonzero brackets are $$[e_1, e_4] = \alpha e_1,\quad [e_2, e_4] = e_2, \quad [e_3, e_4] = e_2 + e_3.$$
We distinguish between two cases associated with this Lie algebra. In the first case, we set $\alpha \neq (0, 1)$, and in the second case, we set $\alpha = 1$.\\
The automorphisms of these two Lie algebras were corrected by the authors of the following paper \cite{christodoulakis2003automorphisms}. The correct results regarding the automorphisms of these two Lie algebras are described in \cite{christodoulakis2003corrigendum}.
\subsubsection{The Lie algebra $\operatorname{A}_{4, 2}^{\alpha}$ where $\alpha \neq (0, 1)$}
The automorphism group of the Lie algebra $\operatorname{A}_{4, 2}^{\alpha}$ where $\alpha \neq (0, 1)$ is formed by elements in $\operatorname{GL}(4, \mathbb{R})$ with the form \cite{christodoulakis2003corrigendum}
$$\begin{bmatrix}
a_1 & 0 & 0 & a_4\\
0 & a_6 & a_7 & a_8\\
0 & 0 & a_6 & a_{12}\\
0 & 0 & 0 & 1
\end{bmatrix}.
$$
\begin{theorem} \label{samestructure}
Every metric on $\operatorname{A}_{4, 2}^{\alpha}$ where $\alpha \neq (0, 1)$ is equivalent to one defined by an orthonormal basis
$$X_1 = e_1, \quad X_2 = b_{12}e_1 + b_{22}e_2, \quad X_3 = b_{13}e_1 + e_3, \quad X_4 = b_{44}e_{4}
$$
where $b_{22}, b_{44} > 0$, \quad $b_{12} \geq 0, \quad b_{13} \in \mathbb{R}$.
\end{theorem}
\begin{proof}
By multiplying $B'$ by an automorphism of $\mathfrak{g}$, our orthonormal basis $B'$ is equivalent to
\begin{equation} \label{form8}
B = \begin{bmatrix}
1 & b_{12} & b_{13} & 0\\
0 & b_{22} & 0 & 0\\
0 & 0 & 1 & 0\\
0 & 0 & 0 & b_{44}\\
\end{bmatrix}, \quad b_{22}, b_{44} > 0, \quad b_{12}, b_{13} \in \mathbb{R}. 
\end{equation}
Given $A \in \operatorname{Aut}(\mathfrak{g})$ satisfying lemma \ref{lemma2.3}, and two orthonormal bases $B = (b_{ij})$ and $C = (c_{ij})$ of the form (\ref{form8}), we calculate
$$C^{-1}AB = \begin{bmatrix}
a_{1} & a_1b_{12} - a_6c_{12}\frac{b_{22}}{c_{22}} & a_{1}b_{13} - \frac{a_7c_{12}}{c_{22}} - a_6c_{13} & 0\\
\\
0 & \frac{a_6b_{22}}{c_{22}} & \frac{a_7}{c_{22}} & 0\\
\\
0 & 0 & a_6 & 0\\
\\
0 & 0 & 0 & \frac{b_{44}}{c_{44}}
\end{bmatrix}.$$
This matrix is orthogonal precesily when $b_{44} = c_{44}$, $a_7 = 0$, $a_1, a_6$ are equal to $\pm1$, $b_{22} = c_{22}$, $a_1b_{12} = a_6c_{12}$ and $a_1b_{13} = a_6c_{13}$. By choice of $a_1$ and $a_6$ we can assume that $b_{12} \geq 0$, but we have $b_{13} \in \mathbb{R}$.
\end{proof}
\subsubsection{The Lie algebra $\operatorname{A}_{4, 2}^{1}$}
The automorphism group of the Lie algebra $\operatorname{A}_{4, 2}^{1}$ is formed by elements in $\operatorname{GL}(4, \mathbb{R})$ with the form \cite{christodoulakis2003corrigendum}
$$\begin{bmatrix}
a_1 & 0 & a_3 & a_4\\
a_5 & a_6 & a_7 & a_8\\
0 & 0 & a_6 & a_{12}\\
0 & 0 & 0 & 1
\end{bmatrix}.
$$
\begin{theorem}
Every metric on $\operatorname{A}_{4, 2}^{1}$ is equivalent to one defined by an orthonormal basis
$$X_1 = e_1, \quad X_2 = b_{12}e_1 + b_{22}e_2, \quad X_3 = e_3, \quad X_4 = b_{44}e_{4}
$$
where $b_{22}, b_{44} > 0, \quad b_{12} \geq 0$.
\end{theorem}
\begin{proof}
By multiplying $B'$ by an upper triangular automorphism of $\mathfrak{g}$, our orthonormal basis $B'$ is equivalent to
\begin{equation} \label{form8'}
B = \begin{bmatrix}
1 & b_{12} & 0 & 0\\
0 & b_{22} & 0 & 0\\
0 & 0 & 1 & 0\\
0 & 0 & 0 & b_{44}\\
\end{bmatrix}, \quad b_{22}, b_{44} > 0. 
\end{equation}
Given two orthonormal bases $B = (b_{ij})$ and $C = (c_{ij})$ of the form (\ref{form8'}), we choose a diagonal automorphism $A$ of $\mathfrak{g}$ and we calculate
$$C^{-1}AB = \begin{bmatrix}
a_{1} & a_1b_{12} - a_6c_{12}\frac{b_{22}}{c_{22}} & 0 & 0\\
\\
0 & \frac{a_6b_{22}}{c_{22}} & 0 & 0\\
\\
0 & 0 & a_6 & 0\\
\\
0 & 0 & 0 & \frac{b_{44}}{c_{44}}
\end{bmatrix}.$$
This matrix is orthogonal precesily when $a_1, a_6$ are equal to $\pm1$, $b_{44} = c_{44}$, $b_{22} = c_{22}$ and $a_1b_{12} = a_6c_{12}$. By choice of $a_1$ and $a_6$ we can put $b_{12} \geq 0$.
\end{proof}
\subsection{The Lie algebra $\operatorname{A}_{4, 3}$}
The Lie algebra $\operatorname{A}_{4, 3}$ has a basis $\mathcal{B} = \left\lbrace e_1, e_2, e_3, e_4\right\rbrace$ such that the nonzero brackets are 
$$[e_1, e_4] = e_1,\quad [e_3, e_4] = e_2.$$
The group of automorphisms of the Lie algebra $\operatorname{A}_{4, 3}$ is given by \cite{christodoulakis2003automorphisms}
$$ \operatorname{Aut}\left( \operatorname{A}_{4, 3}\right) = \left\lbrace \begin{bmatrix}
a_1 & 0 & 0 & a_4\\
0 & a_6 & a_7 & a_8\\
0 & 0 & a_6 & a_{12}\\
0 & 0 & 0 & 1
\end{bmatrix} \right\rbrace  \subset \operatorname{GL}(4, \mathbb{R}).
$$
\begin{theorem}
Every metric on $\operatorname{A}_{4, 3}$ is equivalent to one defined by an orthonormal basis
$$X_1 = e_1, \quad X_2 = b_{12}e_1 + b_{22}e_2, \quad X_3 = b_{13}e_1 + e_3, \quad X_4 = b_{44}e_{4}
$$
where $b_{22}, b_{44} > 0, \quad b_{12} \geq 0, \quad b_{13} \in \mathbb{R}$.
\end{theorem}
\begin{proof}
The proof of this theorem is similar to the one of theorem \ref{samestructure}, because the automorphism group of the Lie algebras $\operatorname{A}_{4, 2}^{\alpha}$ and $\operatorname{A}_{4, 3}$ have the same structure.
\end{proof}
\subsection{The Lie algebra $\operatorname{A}_{4, 4}$}
The Lie algebra $\operatorname{A}_{4, 4}$ has a basis $\mathcal{B} = \left\lbrace e_1, e_2, e_3, e_4\right\rbrace$ such that the nonzero brackets are 
$$[e_1, e_4] = e_1, \quad [e_2, e_4] = e_1 + e_2, \quad [e_3, e_4] = e_2 + e_3.$$
The automorphism group of the Lie algebra $\operatorname{A}_{4, 4}$ is formed by  elements in $\operatorname{GL}(4, \mathbb{R})$ of the following form \cite{christodoulakis2003automorphisms}
$$\begin{bmatrix}
a_1 & a_2 & a_3 & a_4\\
0 & a_1 & a_2 & a_8\\
0 & 0 & a_1 & a_{12}\\
0 & 0 & 0 & 1
\end{bmatrix}.
$$
\begin{theorem}
Every metric on $\operatorname{A}_{4, 4}$ is equivalent to one defined by an orthonormal basis
$$X_1 = e_1, \quad X_2 = b_{12}e_1 + b_{22}e_2, \quad X_3 = b_{33}e_3, \quad X_4 = b_{44}e_{4}
$$
where $b_{22}, b_{33}, b_{44} > 0, \quad b_{12} \in \mathbb{R}$.
\end{theorem}
\begin{proof}
By multiplication by an automorphism of $\mathfrak{g}$, the orthonormal basis $B'$ is equivalent to the following
\begin{equation} \label{form10}
B = \begin{bmatrix}
1 & b_{12} & 0 & 0\\
0 & b_{22} & 0 & 0\\
0 & 0 & b_{33} & 0\\
0 & 0 & 0 & b_{44}\\
\end{bmatrix}, \quad b_{22}, b_{33}, b_{44} > 0, \quad b_{12} \in \mathbb{R}. 
\end{equation}
Given $A \in \operatorname{Aut}(\mathfrak{g})$ satisfying lemma \ref{lemma2.3}, and two orthonormal bases $B = (b_{ij})$ and $C = (c_{ij})$ of the form (\ref{form10}), we calculate
$$C^{-1}AB = \begin{bmatrix}
a_{1} & a_1b_{12} + (a_2 - \frac{a_1c_{12}}{c_{22}})b_{22} & (a_3 - \frac{a_2c_{12}}{c_{22}})b_{33} & 0\\
\\
0 & \frac{a_1b_{22}}{c_{22}} & \frac{a_2b_{33}}{c_{22}} & 0\\
\\
0 & 0 & \frac{a_1b_{33}}{c_{33}} & 0\\
\\
0 & 0 & 0 & \frac{b_{44}}{c_{44}}
\end{bmatrix}.$$
This matrix is orthogonal precesily when $a_2 = a_3 = 0$, $a_1 = \pm1$, $b_{22} = c_{22}$, $b_{33} = c_{33}$, $b_{44} = c_{44}$ and $b_{12} = c_{12}$.
\end{proof}
\subsection{The Lie algebra $\operatorname{A}_{4, 5}^{\alpha, \beta}$}
The Lie algebra $\operatorname{A}_{4, 5}^{\alpha, \beta}$ has a basis $\mathcal{B} = \left\lbrace e_1, e_2, e_3, e_4\right\rbrace$ such that the nonzero brackets are 
$$[e_1, e_4] = e_1, \quad [e_2, e_4] = \alpha e_2, \quad [e_3, e_4] = \beta e_3.$$
The automorphism group of the Lie algebra $\operatorname{A}_{4, 5}^{\alpha, \beta}$ consists of elements of $\operatorname{GL}(4, \mathbb{R})$ of the form \cite{christodoulakis2003automorphisms}
$$\begin{bmatrix}
a_1 & 0 & 0 & a_4\\
0 & a_6 & 0 & a_8\\
0 & 0 & a_{11} & a_{12}\\
0 & 0 & 0 & 1
\end{bmatrix}.
$$
\begin{theorem}
Every metric on $\operatorname{A}_{4, 5}^{\alpha, \beta}$ is equivalent to one defined by an orthonormal basis
$$X_1 = e_1, \quad X_2 = b_{12}e_1 + e_2, \quad X_3 = b_{13}e_1 + b_{23}e_2 + e_3, \quad X_4 = b_{44}e_{4}
$$
where  $b_{44} > 0$, \quad $b_{12}, b_{13} \geq 0$, \quad $b_{23} \in \mathbb{R}$.
\end{theorem}
\begin{proof}
Applying an automorphism of $\mathfrak{g}$,  we can put $B'$ in the form
\begin{equation} \label{form11}
B = \begin{bmatrix}
1 & b_{12} & b_{13} & 0\\
0 & 1 & b_{23} & 0\\
0 & 0 & 1 & 0\\
0 & 0 & 0 & b_{44}\\
\end{bmatrix}, \quad b_{44} > 0. 
\end{equation}
Let $A$ be an automorphism of $\mathfrak{g}$ satisfying lemma \ref{lemma2.3}, given two orthonormal bases $B = (b_{ij})$ and $C = (c_{ij})$ of the form (\ref{form11}), we obtain that
$$C^{-1}AB = \begin{bmatrix}
a_{1} & x & y & 0\\
0 & a_{6} & z & 0\\
0 & 0 & a_{11} & 0\\
0 & 0 & 0 & \frac{b_{44}}{c_{44}}
\end{bmatrix}.$$
Where the elements $x, y$ and $z$ are given by
$$
\left\lbrace\begin{array}{lll}
x = a_1b_{12} - a_{6}c_{12} \\
y = a_1b_{13} - a_{6}c_{12}b_{23} + (c_{12}c_{23} - c_{13})a_{11} \\
z = a_{6}b_{23} - a_{11}c_{23}
\end{array}\right.
$$
The matrix $C^{-1}AB$ is orthogonal precisely when $b_{44} = c_{44}$, $a_{1}, a_{6}$ and $a_{11}$ are equal to $\pm1$ and $x = y = z = 0$. Thus we obtain that
$$
\left\lbrace\begin{array}{lll}
a_1b_{12} =  a_{6}c_{12} \\
a_{6}b_{23} = a_{11}c_{23}\\
a_1b_{13} = a_{11}c_{13}
\end{array}\right.
$$
By choice of $a_{1}$, $a_6$ and $a_{11}$, any orthonormal basis is equivalent to one with $b_{12}, b_{13} \geq 0$ and $b_{23} \in \mathbb{R}$.
\end{proof}
\subsection{The Lie algebra $\operatorname{A}_{4, 6}^{\alpha, \beta}$}
The Lie algebra $\operatorname{A}_{4, 6}^{\alpha, \beta}$ has a basis $\mathcal{B} = \left\lbrace e_1, e_2, e_3, e_4\right\rbrace$ such that the nonzero brackets are 
$$[e_1, e_4] = \alpha e_1, \quad [e_2, e_4] = \beta e_2 - e_3, \quad [e_3, e_4] = e_2 + \beta e_3.$$
The automorphisms of the Lie algebra $\operatorname{A}_{4, 6}^{\alpha, \beta}$ are given by elements of $\operatorname{GL}(4, \mathbb{R})$ that take the form \cite{christodoulakis2003automorphisms}
$$\begin{bmatrix}
a_1 & 0 & 0 & a_4\\
0 & a_6 & a_7 & a_8\\
0 & -a_7 & a_6 & a_{12}\\
0 & 0 & 0 & 1
\end{bmatrix}.
$$
\begin{theorem}
Every metric on $\operatorname{A}_{4, 6}^{\alpha, \beta}$ is equivalent to one defined by an orthonormal basis
$$X_1 = e_1, \quad X_2 = b_{12}e_1 + e_2, \quad X_3 = b_{13}e_1 + b_{33}e_3, \quad X_4 = b_{44}e_{4}$$
$\hspace{2.3cm}  b_{44} > 0, \quad 0 < b_{33} < 1, \quad b_{12}, b_{13} \geq 0.$\\
or 
$$X_1 = e_1, \quad X_2 = b_{12}e_1 + e_2, \quad X_3 = b_{13}e_1 + b_{33}e_3, \quad X_4 = b_{44}e_{4}$$
$\hspace{2.3cm}  b_{44} > 0, \quad 0 < b_{33} < 1, \quad b_{12} > 0, \quad b_{13} < 0.$\\
or
$$X_1 = e_1, \quad X_2 = b_{12}e_1 + e_2, \quad X_3 = e_3, \quad X_4 = b_{44}e_{4}$$
$\hspace{2.3cm}  b_{44} > 0, \quad b_{12} \geq 0.$
\end{theorem}
\begin{proof}
Applying an upper triangular automorphism of $\mathfrak{g}$,  we can put $B'$ in the form
$$
B'' = \begin{bmatrix}
1 & b_{12} & b_{13} & 0\\
0 & 1 & b_{23} & 0\\
0 & 0 & b_{33} & 0\\
0 & 0 & 0 & b_{44}\\
\end{bmatrix}, \quad b_{33}, b_{44} > 0. 
$$
By lemma \ref{lemma2.5}, $B''$ is equivalent to the following
\begin{equation} \label{form12}
B = \begin{bmatrix}
1 & b_{12} & b_{13} & 0\\
0 & 1 & 0 & 0\\
0 & 0 & b_{33} & 0\\
0 & 0 & 0 & b_{44}\\
\end{bmatrix}, \quad 0 < b_{33} \leq 1, \quad b_{44} > 0. 
\end{equation}
Let $A$ be an automorphism of $\mathfrak{g}$ satisfying lemma \ref{lemma2.3}, given two orthonormal bases $B = (b_{ij})$ and $C = (c_{ij})$ of the form (\ref{form12}), we have
$$C^{-1}AB = \begin{bmatrix}
a_1 & b_{12} - a_6c_{12} + a_7\frac{c_{13}}{c_{33}} & b_{13} + \left( -a_7c_{12} - a_6\frac{c_{13}}{c_{33}}\right)b_{33} & 0\\
\\
0 & a_6 & a_7b_{33} & 0\\
\\
0 & \frac{-a_7}{c_{33}} & \frac{a_6b_{33}}{c_{33}} & 0\\
\\
0 & 0 & 0 & \frac{b_{44}}{c_{44}}
\end{bmatrix}.$$
This matrix will be orthogonal precesily when its middle $2 \times 2$ block is orthogonal,\\ $b_{44} = c_{44}$, $a_{1} = \pm 1$ and the other elements are zero.\\
If $b_{33} \neq c_{33}$, then the middle $2 \times 2$ block cannot be orthogonal because we have\\ $0 < b_{33}, c_{33} \leq 1$. Thus, we assume that $b_{33} = c_{33}$.\\
First suppose that $b_{33} = c_{33} = 1$, then we have 
$\begin{bmatrix}
b_{12}\\
b_{13}
\end{bmatrix} = \begin{bmatrix}
a_6 & -a_7\\
a_7 & a_6
\end{bmatrix}\begin{bmatrix}
c_{12}\\
c_{13}
\end{bmatrix}$. This means that $(b_{12}, b_{13})$ and $(c_{12}, c_{13})$ are in the same orbit of the natural action of $\operatorname{SO}(2, \mathbb{R})$ on $\mathbb{R}^2$. From this we infer that when $b_{33} = 1$, we can put $b_{13} = 0$ and $b_{12} \geq 0$. This is the third metric in our theorem.\\
If $0 < b_{33} < 1$, then the condition that the middle $2 \times 2$ block of $C^{-1}AB$ is orthogonal implies that $a_7 = 0$. Thus
$ \begin{bmatrix}
b_{12}\\
b_{13}
\end{bmatrix} = \pm \begin{bmatrix}
c_{12}\\
c_{13}
\end{bmatrix}.$\\
If $c_{12}$ and $c_{13}$ have the same sign, we can take $b_{12}, b_{13} \geq 0$ (first metric in our theorem). If  $c_{12}$ and $c_{13}$ have different sign, we can take $b_{12} > 0$, $b_{13} < 0$ (second metric in our theorem).
\end{proof}
\subsection{The Lie algebra $\operatorname{A}_{4, 7}$}
The Lie algebra $\operatorname{A}_{4, 7}$ has a basis $\mathcal{B} = \left\lbrace e_1, e_2, e_3, e_4\right\rbrace$ such that the nonzero brackets are 
$$[e_2, e_3] = e_1, \quad [e_1, e_4] = 2 e_1, \quad [e_2, e_4] = e_2, \quad [e_3, e_4] = e_2 + e_3.$$
The automorphism group of the Lie algebra $\operatorname{A}_{4, 7}$ is formed by elements in $\operatorname{GL}(4, \mathbb{R})$ with the form \cite{christodoulakis2003automorphisms}
$$\begin{bmatrix}
a_6^2 & -a_{12}a_6 & -a_{12}(a_6 + a_7) + a_6a_8& a_4\\
0 & a_6 & a_7 & a_8\\
0 & 0 & a_6 & a_{12}\\
0 & 0 & 0 & 1
\end{bmatrix}.
$$
\begin{theorem} \label{theorem4.7}
Every metric on $\operatorname{A}_{4, 7}$ is equivalent to one defined by an orthonormal basis
$$X_1 = b_{11}e_1, \quad X_2 = b_{12}e_1 + e_2, \quad X_3 = b_{13}e_1 + b_{33}e_3, \quad X_4 = b_{44}e_{4}
$$
where  $b_{11}, b_{33}, b_{44} > 0, \quad b_{12} \geq 0, \quad b_{13} \in \mathbb{R}$.
\end{theorem}
\begin{proof}
By multiplying $B'$ by an automorphism of $\mathfrak{g}$,  we can put $B'$ in the form
\begin{equation} \label{form13}
B = \begin{bmatrix}
b_{11} & b_{12} & b_{13} & 0\\
0 & 1 & 0 & 0\\
0 & 0 & b_{33} & 0\\
0 & 0 & 0 & b_{44}\\
\end{bmatrix}, \quad b_{11}, b_{33}, b_{44} > 0. 
\end{equation}
Let $A$ be an automorphism of $\mathfrak{g}$ satisfying lemma \ref{lemma2.3}, given two orthonormal bases $B = (b_{ij})$ and $C = (c_{ij})$ of the form (\ref{form13}), we have
$$C^{-1}AB = \begin{bmatrix}
\frac{a_6^2b_{11}}{c_{11}} & \frac{a_6^2b_{12} - a_6c_{12}}{c_{11}} & \frac{a_6^2b_{13}}{c_{11}} + \left( \frac{-a_7c_{12}}{c_{11}} - \frac{-a_6c_{13}}{c_{11}c_{33}}\right)b_{33} & 0\\
\\
0 & a_6 & a_7b_{33} & 0\\
\\
0 & 0 & \frac{a_6b_{33}}{c_{33}} & 0\\
\\
0 & 0 & 0 & \frac{b_{44}}{c_{44}}
\end{bmatrix}.$$
This matrix is orthogonal precesily when $a_7 = 0$, $a_6 = \pm1$, $b_{44} = c_{44}$, $b_{33} = c_{33}$, $b_{11} = c_{11}$, $b_{12} = a_6c_{12}$ and $b_{13} = a_6c_{13}$. Hence by choice of $a_6$ we can put $b_{12} \geq 0$ and thus $b_{13} \in \mathbb{R}$.
\end{proof}
\subsection{The Lie algebra $\operatorname{A}_{4, 9}^{\beta}$}
The Lie algebra $\operatorname{A}_{4, 9}^{\beta}$ has a basis $\mathcal{B} = \left\lbrace e_1, e_2, e_3, e_4\right\rbrace$ such that the nonzero brackets are 
$$[e_2, e_3] = e_1, \quad [e_1, e_4] = (1 + \beta) e_1, \quad [e_2, e_4] = e_2, \quad [e_3, e_4] = \beta e_3.$$
The automorphism group of the Lie algebra $\operatorname{A}_{4, 9}^{\beta}$ consists of elements of $\operatorname{GL}(4, \mathbb{R})$ of the form \cite{christodoulakis2003automorphisms}
$$\begin{bmatrix}
a_{11}a_6 & -a_{12}a_6/\beta & a_8a_{11}& a_4\\
0 & a_6 & 0 & a_8\\
0 & 0 & a_{11} & a_{12}\\
0 & 0 & 0 & 1
\end{bmatrix}.
$$
\begin{theorem}
Every metric on $\operatorname{A}_{4, 9}^{\beta}$ is equivalent to one defined by an orthonormal basis
$$X_1 = b_{11}e_1, \quad X_2 = b_{12}e_1 + e_2, \quad X_3 = b_{13}e_1 + b_{23}e_2 + e_3, \quad X_4 = b_{44}e_{4}
$$
where  $b_{11}, b_{44} > 0$, \quad $b_{12}, b_{13} \geq 0, \quad b_{23} \in \mathbb{R}$.
\end{theorem}
\begin{proof}
We can find $A \in \operatorname{Aut}(\mathfrak{g})$, with $AB' = B$, where
\begin{equation} \label{form14}
B = \begin{bmatrix}
b_{11} & b_{12} & b_{13} & 0\\
0 & 1 & b_{23} & 0\\
0 & 0 & 1 & 0\\
0 & 0 & 0 & b_{44}\\
\end{bmatrix}, \quad b_{11}, b_{44} > 0. 
\end{equation}
Let $A$ be an automorphism of $\mathfrak{g}$ satisfying lemma \ref{lemma2.3}, given two orthonormal bases $B = (b_{ij})$ and $C = (c_{ij})$ of the form (\ref{form14}), we have
$$C^{-1}AB = \begin{bmatrix}
\frac{a_{11}a_6b_{11}}{c_{11}} & \frac{x}{c_{11}} & \frac{y}{c_{11}} & 0\\
\\
0 & a_{6} & z & 0\\
\\
0 & 0 & a_{11} & 0\\
\\
0 & 0 & 0 & \frac{b_{44}}{c_{44}}
\end{bmatrix}.$$
Where the elements $x, y$ and $z$ are given by
$$
\left\lbrace\begin{array}{lll}
x = a_6a_{11}b_{12} - a_{6}c_{12} \\
y = a_6a_{11}b_{13} - a_{6}c_{12}b_{23} + (c_{12}c_{23} - c_{13})a_{11} \\
z = a_{6}b_{23} - a_{11}c_{23}
\end{array}\right.
$$
This matrix is orthogonal precesily when $a_6$, $a_{11}$ are equal to $\pm1$, $b_{11} = c_{11}$, $b_{44} = c_{44}$ and $x = y = z = 0$. This means that 
$$
\left\lbrace\begin{array}{lll}
a_{11}b_{12} =  c_{12} \\
a_{6}b_{13} = c_{13}\\
a_6b_{23} = a_{11}c_{23}
\end{array}\right.
$$
By choice of $a_6$ and $a_{11}$ we can assume that $b_{12}, b_{13} \geq 0$ and thus $b_{23} \in \mathbb{R}$.
\end{proof}
\subsection{The Lie algebra $\operatorname{A}_{4, 11}^{\alpha}$}
The Lie algebra $\operatorname{A}_{4, 11}^{\alpha}$ has a basis $\mathcal{B} = \left\lbrace e_1, e_2, e_3, e_4\right\rbrace$ such that the nonzero brackets are 
$$[e_2, e_3] = e_1, \quad [e_1, e_4] = 2\alpha e_1, \quad [e_2, e_4] = \alpha e_2 - e_3, \quad [e_3, e_4] = e_2 + \alpha e_3.$$
The automorphism group associated with the Lie algebra $\operatorname{A}_{4, 11}^{\alpha}$ is comprised of those elements in $\operatorname{GL}(4, \mathbb{R})$ which are of the form \cite{christodoulakis2003automorphisms}
$$\begin{bmatrix}
a_6^2 + a_7^2 & -(w_1)/(1 + \alpha^2) & -(w_2)/(1 + \alpha^2) & a_4\\
\\
0 & a_6 & a_7 & a_8\\
\\
0 & -a_7 & a_6 & a_{12}\\
\\
0 & 0 & 0 & 1
\end{bmatrix}.
$$
where $w_1 = a_6(\alpha a_{12} + a_8) + a_7(\alpha a_8 - a_{12})$ and $w_2 = a_6(a_{12} - \alpha a_8) + a_7(\alpha a_{12} + a_{8})$.
\begin{theorem}
Every metric on $\operatorname{A}_{4, 11}^{\alpha}$ is equivalent to one defined by an orthonormal basis
$$X_1 = b_{11}e_1, \quad X_2 = b_{12}e_1 + e_2, \quad X_3 = b_{13}e_1 + b_{33}e_3, \quad X_4 = b_{44}e_{4}$$
$\hspace{2.3cm}  b_{11}, b_{44} > 0, \quad 0 < b_{33} < 1, \quad b_{12}, b_{13} \geq 0.$\\
or 
$$X_1 = b_{11}e_1, \quad X_2 = b_{12}e_1 + e_2, \quad X_3 = b_{13}e_1 + b_{33}e_3, \quad X_4 = b_{44}e_{4}$$
$\hspace{2.3cm}  b_{11}, b_{44} > 0, \quad 0 < b_{33} < 1, \quad b_{12} > 0, \quad b_{13} < 0.$\\
or
$$X_1 = b_{11}e_1, \quad X_2 = b_{12}e_1 + e_2, \quad X_3 = e_3, \quad X_4 = b_{44}e_{4}$$
$\hspace{2.3cm}  b_{11}, b_{44} > 0, \quad b_{12} \geq 0.$
\end{theorem}
\begin{proof}
By multiplying $B'$ by an upper triangular automorphism of $\mathfrak{g}$,  we can put $B'$ in the form
$$
B'' = \begin{bmatrix}
b_{11} & b_{12} & b_{13} & 0\\
0 & 1 & b_{23} & 0\\
0 & 0 & b_{33} & 0\\
0 & 0 & 0 & b_{44}\\
\end{bmatrix}, \quad b_{11}, b_{33}, b_{44} > 0. 
$$
By lemma \ref{lemma2.5}, $B''$ is equivalent to the following
\begin{equation} \label{form141}
B = \begin{bmatrix}
b_{11} & b_{12} & b_{13} & 0\\
0 & 1 & 0 & 0\\
0 & 0 & b_{33} & 0\\
0 & 0 & 0 & b_{44}\\
\end{bmatrix}, \quad b_{11}, b_{44} > 0, \quad 0 < b_{33} \leq 1. 
\end{equation}
Let $A$ be an automorphism of $\mathfrak{g}$ satisfying lemma \ref{lemma2.3} (and such that $a_6^2 + a_7^2 = 1$), given two orthonormal bases $B = (b_{ij})$ and $C = (c_{ij})$ of the form (\ref{form141}), we have
$$C^{-1}AB = \begin{bmatrix}
\frac{b_{11}}{c_{11}} & \frac{b_{12} - a_6c_{12} + a_7\frac{c_{13}}{c_{33}}}{c_{11}} & \frac{b_{13}}{c_{11}} + \left( \frac{-a_7c_{12} - a_6\frac{c_{13}}{c_{33}}}{c_{11}}\right)b_{33} & 0\\
\\
0 & a_6 & a_7b_{33} & 0\\
\\
0 & \frac{-a_7}{c_{33}} & \frac{a_6b_{33}}{c_{33}} & 0\\
\\
0 & 0 & 0 & \frac{b_{44}}{c_{44}}
\end{bmatrix}.$$
Using the same reasoning as in the proof of theorem \ref{reasoning}, we obtain the metrics described in the theorem.
\end{proof}
\subsection{The Lie algebra $\operatorname{A}_{4, 12}$}
The Lie algebra $\operatorname{A}_{4, 12}$ has a basis $\mathcal{B} = \left\lbrace e_1, e_2, e_3, e_4\right\rbrace$ such that the only nonzero brackets are 
$$[e_1, e_3] = e_1, \quad [e_2, e_3] = e_2, \quad [e_1, e_4] = -e_2, \quad [e_2, e_4] = e_1.$$
The automorphism group of the Lie algebra $\operatorname{A}_{4, 12}$ is the group of all elements of $\operatorname{GL}(4, \mathbb{R})$ of the form \cite{christodoulakis2003automorphisms}
$$\begin{bmatrix}
a_1 & a_2 & a_3 & a_4\\
-a_2 & a_1 & a_4 & -a_3\\
0 & 0 & 1 & 0\\
0 & 0 & 0 & 1
\end{bmatrix} \quad \text{or}\quad \begin{bmatrix}
a_1 & a_2 & a_3 & a_4\\
a_2 & -a_1 & -a_4 & a_3\\
0 & 0 & 1 & 0\\
0 & 0 & 0 & -1
\end{bmatrix}.
$$
\begin{theorem}
Every metric on $\operatorname{A}_{4, 12}$ is equivalent to one defined by an orthonormal basis
$$X_1 = e_1, \quad X_2 = b_{22}e_2, \quad X_3 = b_{33}e_3, \quad X_4 = b_{14}e_1 + b_{24}e_2 + b_{34}e_3 + b_{44}e_{4}
$$
$\hspace{2.3cm}  b_{33}, b_{44} > 0, \quad  0 < b_{22} < 1, \quad b_{14}, b_{24} \geq 0, \quad b_{34} \in \mathbb{R}$.\\
or 
$$X_1 = e_1, \quad X_2 = b_{22}e_2, \quad X_3 = b_{33}e_3, \quad X_4 = b_{14}e_1 + b_{24}e_2 + b_{34}e_3 + b_{44}e_{4}
$$
$\hspace{2.3cm}  b_{33}, b_{44} > 0, \quad  0 < b_{22} < 1, \quad b_{14} > 0, \quad b_{24} < 0, \quad b_{34} \in \mathbb{R}$.\\
or
$$X_1 = e_1, \quad X_2 = e_2, \quad X_3 = b_{33}e_3, \quad X_4 = b_{14}e_1 + b_{34}e_3 + b_{44}e_{4}
$$
$\hspace{2.3cm}  b_{33}, b_{44} > 0, \quad b_{14} \geq 0, \quad b_{34} \in \mathbb{R}$.
\end{theorem}
\begin{proof}
By multiplication by an upper triangular automorphism of $\mathfrak{g}$, we can put $B'$ in the form
$$
B'' = \begin{bmatrix}
1 & b_{12} & 0 & b_{14}\\
0 & b_{22} & 0 & b_{24}\\
0 & 0 & b_{33} & b_{34}\\
0 & 0 & 0 & b_{44}\\
\end{bmatrix}, \quad b_{22}, b_{33}, b_{44} > 0. 
$$
By lemma \ref{lemma2.5}, $B''$ is equivalent to the following
\begin{equation} \label{form16}
B = \begin{bmatrix}
1 & 0 & 0 & b_{14}\\
0 & b_{22} & 0 & b_{24}\\
0 & 0 & b_{33} & b_{34}\\
0 & 0 & 0 & b_{44}\\
\end{bmatrix}, \quad b_{33}, b_{44} > 0, \quad 0 < b_{22} \leq 1. 
\end{equation}
Given two orthonormal bases $B = (b_{ij})$ and $C = (c_{ij})$ of the form (\ref{form16}), consider $A$ to be an automorphism of $\mathfrak{g}$ of the form $\begin{bmatrix}
a_1 & a_2 & 0 & 0\\
-a_2 & a_1 & 0 & 0\\
0 & 0 & 1 & 0\\
0 & 0 & 0 & 1
\end{bmatrix}$. Then we calculate
$$C^{-1}AB = \begin{bmatrix}
a_1 & a_2b_{22} & 0 & x\\
\\
\frac{-a_2}{c_{22}} & \frac{a_1b_{22}}{c_{22}} & 0 & \frac{y}{c_{22}}\\
\\
0 & 0 & \frac{b_{33}}{c_{33}} & \frac{z}{c_{33}}\\
\\
0 & 0 & 0 & \frac{b_{44}}{c_{44}}
\end{bmatrix}.$$
Where the elements $x, y$ and $z$ are given by
$$
\left\lbrace\begin{array}{lll}
x = a_1b_{14} + a_2b_{24} - c_{14}\frac{b_{44}}{c_{44}} \\
y = -a_2b_{14} + a_1b_{24} - c_{24}\frac{b_{44}}{c_{44}} \\
z = b_{34} - c_{34}\frac{b_{44}}{c_{44}}
\end{array}\right.
$$
The matrix $C^{-1}AB$ is orthogonal precisely when the upper $2 \times 2$ block is orthogonal, $x = y = z = 0$, $b_{33} = c_{33}$ and $b_{44} = c_{44}$. If $b_{22} \neq c_{22}$, then the upper $2 \times 2$ block cannot be orthogonal because we have $0 < b_{22}, c_{22} \leq 1$. Thus, we assume that $b_{22} = c_{22}$.\\
First suppose that $b_{22} = c_{22} = 1$, then we have 
$$\begin{bmatrix}
x\\
y
\end{bmatrix} = \begin{bmatrix}
0\\
0
\end{bmatrix} \Leftrightarrow \begin{bmatrix}
a_1 & a_2\\
-a_2 & a_1
\end{bmatrix}\begin{bmatrix}
b_{14}\\
b_{24}
\end{bmatrix} = \begin{bmatrix}
c_{14}\\
c_{24}
\end{bmatrix}.$$
This means that $(b_{14}, b_{24})$ and $(c_{14}, c_{24})$ are in the same orbit of the natural action of $\operatorname{SO}(2, \mathbb{R})$ on $\mathbb{R}^2$. From this we infer that when $b_{22} = 1$, we can put $b_{24} = 0$ and $b_{14} \geq 0$. This is the third metric in our theorem.\\
If $0 < b_{22} < 1$, then the condition that the upper $2 \times 2$ block of $C^{-1}AB$ is orthogonal implies that $a_2 = 0$. In this case we have 
$$\begin{bmatrix}
x\\
y
\end{bmatrix} = \begin{bmatrix}
0\\
0
\end{bmatrix} \Leftrightarrow \begin{bmatrix}
b_{14}\\
b_{24}
\end{bmatrix} = \pm \begin{bmatrix}
c_{14}\\
c_{24}
\end{bmatrix}.$$
If $c_{14}$ and $c_{24}$ have the same sign, we can take $b_{14}, b_{24} \geq 0$ (first metric in our theorem). If  $c_{14}$ and $c_{24}$ have different sign, we can take $b_{14} > 0$, $b_{24} < 0$ (second metric in our theorem).\\
If we choose an automorphism of the second form in $\operatorname{Aut}(\mathfrak{g})$, using the fact that a complete set of orbit representatives for the natural action of $\operatorname{O}(2, \mathbb{R})$ on $\mathbb{R}^2$ is the nonnegative $x$-axis, we obtain a result similar to that of the previous case.
\end{proof}

\section*{Acknowledgement}

The authors express their heartfelt gratitude to professor Mohamed Boucetta for his invaluable encouragement in bringing this work to fruition.

\end{document}